\documentclass[12pt]{article}
\DeclareMathAlphabet\mathbfcal{OMS}{cmsy}{b}{n}
\usepackage[dvipsnames]{xcolor}
\usepackage{enumerate}
\usepackage[utf8]{inputenc}
\usepackage{amsfonts}
\usepackage{amsthm}
\usepackage{amsmath}
\usepackage{lscape}
\usepackage{mathrsfs}
\usepackage{graphicx}
\usepackage{tikz-cd}
\usepackage{calc}
\usepackage{float}
\usepackage{multirow}
\usepackage{lipsum}
\usepackage{amssymb}
\usepackage{mathrsfs}
\usepackage{mathtools}
\usepackage{mathdots}
\usepackage{fancyhdr}
\usepackage{tikz}
\usepackage{faktor}
\usepackage{setspace}
\usetikzlibrary{decorations.markings}

\theoremstyle{plain}
\usepackage{sectsty}
\usepackage{kpfonts}
\sectionfont{\fontsize{14}{15}\selectfont}
\subsectionfont{\fontsize{12}{15}\selectfont}

\newtheorem{lemma}{Lemma}

\newtheorem{proposition}{Proposition}

\newcommand{\End}{{\mbox{End}}}

\newcommand{\Her}{{\mbox{Herm}}}

\newcommand{\vol}{{\mbox{vol}}}

\newtheorem*{observation*}{Observation}
\newtheorem*{theorem*}{Theorem}
\newtheorem*{claim*}{Claim}  
\usepackage{blindtext}
\title{\textbf{Functionals on the space of almost complex structures}}
\author{Gabriella Clemente}
\date{}
\allowdisplaybreaks
\begin{document}
\maketitle
\begin{abstract}
We study functionals on the space of almost complex structures on a compact $\mathbb{C}$-manifold, whose variational properties could be used to tackle Yau's Challenge.
\end{abstract}
\tableofcontents
\addcontentsline{toc}{section}{Introduction}
\section*{Introduction}

This is supposed to be a step in the direction of understanding Yau's Challenge, which is to determine if there are compact almost complex manifolds of dimension at least $3$ that cannot be given an integrable almost complex structure \cite{YC1}, through the calculus of variations. S-T.\ Yau proposed devising a parabolic flow on the space of almost complex structures to study this question \cite{YC}. 

Let $X$ be a real $2n$-dimensional compact manifold, and $AC(X)=\{J \in C^{\infty}(X,\End_{\mathbb{C}}(T_X)) \mid J^2=-Id\}$ be the space of almost complex structures on $X.$ This is an almost complex Fr{\'e}chet manifold, and for any $J \in AC(X),$ $T_{Ac(X),J}=\{h \in C^{\infty}(X,\End_{\overline{\mathbb{C}}}(T_X)) \mid J \circ h + h \circ J=0\},$ which can be seen from the identity $0=dJ^2=dJ \circ J + J \circ dJ.$ An almost complex structure $\mathcal{J}:AC(X) \to \End(T_{AC(X)})$ is given as $\mathcal{J}(J)(u)=J \circ u,$ for any $J \in AC(X)$ and $u \in T_{AC(X),J}.$ Let $g$ be a fixed Riemannian metric on $X,$ and note that for any $J \in AC(X),$ we get an almost hermitian metric $g_J:=\frac{1}{2}\big(g(\cdot,\cdot)+g(J \cdot,J \cdot)\big).$ We are looking for an energy functional $\mathcal{F}$ on $AC(X)$ whose associated gradient flow is a parabolic PDE. Ideally, the critical points of $\mathcal{F}$ should be the integrable almost complex structures on $X,$ and the Euler-Lagrange equation of $\mathcal{F}$ should be elliptic so that the complex structures on $X$ are energy minimizers. We would then expect any solution of the flow equation of $\mathcal{F}$ to converge to a genuine complex structure on $X.$ In some special cases, such as when $AC(X)$ is connected (e.g.\ $AC(S^6)$), the non-existence of a flow solution might translate to the non-existence of complex structures. A more thorough development of these ideas will be the subject of future research. Here we only derive the Euler-Lagrange equations of the functionals $\mathcal{N},\widetilde{\mathcal{N}}:AC(X)\to \mathbb{R}_{\geq 0},$  \[\mathcal{N}(J):=\int_X \parallel N_J \parallel^2_{g_J} \vol_g,\mbox{ and } \widetilde{\mathcal{N}}(J):=\int_X \parallel N_J \parallel^2_{g_J} \vol_{g_J},\] where $\vol_g$ is the Riemannian volume form, and $\vol_{g_J}$ is the volume form of $\omega:=\frac{i}{2}(g_J - \overline{g_J}).$ Note that both $\mathcal{N},$ and $\widetilde{\mathcal{N}}$ are identically zero on the integrable structures. A very similar, real version of $\mathcal{N}$ appears in \cite{mil}.

We can think of first variations in terms of linear approximations. Let $\Her\big(\Lambda^2 T^{{0,1}^*}_X \otimes T^{1,0}_X \big)$ denote the space of hermitian metrics on $\Lambda^2 T^{{0,1}^*}_X \otimes T^{1,0}_X.$ Let \[f:AC(X) \to C^{\infty}\big(X,\Lambda^2 T^{{0,1}^*}_X \otimes T^{1,0}_X\big) \times \Her\big(\Lambda^2 T^{{0,1}^*}_X \otimes T^{1,0}_X \big)\] be the function $f(I)=(N,h)=(N(I),h(I))=(N_I,\overline{{g_I}^{-1}} \wedge \overline{{g_I}^{-1}} \otimes g_I),$ and \[\phi:C^{\infty}\big(X,\Lambda^2 T^{{0,1}^*}_X \otimes T^{1,0}_X \big) \times \Her\big(\Lambda^2 T^{{0,1}^*}_X \otimes T^{1,0}_X \big) \to C^{\infty}(X,\mathbb{R}_{\geq 0})\] be the function $\phi(N,h)=\overline{h^{-1}} \wedge \overline{h^{-1}}\otimes h(N,N),$ and define $\psi:=\phi \circ f,$ where $\psi(I)=\overline{{g_I}^{-1}} \wedge \overline{{g_I}^{-1}} \otimes g_I (N_I,N_I)=:\|N_I\|^2_{g_I}.$ And now, let \[F:AC(X) \to C^{\infty}\big(X,\Lambda^2 T^{{0,1}^*}_X \otimes T^{1,0}_X\big) \times \Her\big(\Lambda^2 T^{{0,1}^*}_X \otimes T^{1,0}_X \big) \times \Omega^{n,n}(X),\] $F(I)=(N(I),h(I),\vol(I))=(N_I,\overline{{g_I}^{-1}} \wedge \overline{{g_I}^{-1}} \otimes g_I,\vol_{g_I}),$ and \[\Phi:C^{\infty}\big(X,\Lambda^2 T^{{0,1}^*}_X \otimes T^{1,0}_X\big) \times \Her\big(\Lambda^2 T^{{0,1}^*}_X \otimes T^{1,0}_X \big) \times \Omega^{n,n}(X) \to C^{\infty}(X,\mathbb{R}_{\geq 0}),\] $\Phi(N,h,\vol)=\overline{h^{-1}} \wedge \overline{h^{-1}}\otimes h(N,N)\vol_h,$ and define $\Psi:=\Phi \circ F$ so that $\Psi(I)=\overline{{g_I}^{-1}} \wedge \overline{{g_I}^{-1}} \otimes g_I (N_I,N_I)\vol_{g_I}=:\|N_I\|^2_{g_I}\vol_{g_I}.$ Let $\gamma=\frac{1}{2}\big(g(u \cdot,J \cdot)+g(J \cdot,u \cdot)\big).$ Let $J \in AC(X),$ and $\delta J$ be a small perturbation of $J;$ i.e.\ if $u$ is a nearby structure in $AC(X),$ then $\delta J=(J+u)-J.$  Let $\delta N=N(J+\delta J)-N(J)=N_{J+u} - N_J = d_J N_J(u)+O(u^2),$ $\delta h=h(J+\delta J)-h(J)=g_{J+u}-g_J = d_J g_J (u)+O(u^2)=\gamma+O(u^2),$ and $d\vol=\vol(J+\delta J)-\vol(J)=\vol_{g_{J+u}}-\vol_{g_J}=d_J(\vol_{g_J})(u)+O(u^2)=d_J(\vol_{g_J})(u).$ Then, 
\begin{equation*}
\begin{split}
d_J (\|N_J\|^2_{g_J})(u)&\approx [\phi(N+\delta N,h + \delta h)-\phi(N,h + \delta h)]+[\phi(N,h + \delta h)-\phi(N,h)]\\
&\approx d_N \phi(N,h + \delta h)\cdot \delta N +d_h \phi(N,h) \cdot \delta h\\
&=d_{N_J}(\|N_J\|^2_{g_{J+u}})\cdot d_J N_J(u)+d_{g_J}(\|N_J\|^2_{g_J})\cdot \gamma,
\end{split}
\end{equation*}
and 
\begin{equation*}
\begin{split}
d_J(\|N_J\|^2_{g_J} \vol{g_J})(u)&\approx [\Phi(N+\delta N,h + \delta h,\vol + \delta \vol)-\phi(N,h + \delta h,\vol + \delta \vol)]+\\
&[\Phi(N,h + \delta h,\vol + \delta \vol)-\Phi(N,h,\vol + \delta \vol)]+\\
&[\Phi(N,h,\vol + \delta \vol)-\Phi(N,h,\vol)]\\
&\approx d_N \Phi(N,h + \delta h,\vol + \delta \vol)\cdot \delta N + d_h \Phi(N,h,\vol + \delta \vol) \cdot \delta h +\\
& d_{\vol}\Phi(N,h,\vol)\cdot \delta \vol\\
&=d_{N_J}(\|N_J\|^2_{g_{J+u}}\vol_{g_{J+u}})\cdot d_J N_J(u)+d_{g_J}(\|N_J\|^2_{g_J}\vol_{g_{J+u}})\cdot\gamma+\\
&d_{\vol}(\|N_J\|^2_{g_J}\vol_{g_J})\cdot d_J(\vol_{g_J})(u).
\end{split}
\end{equation*} 

The Nijenhuis tensor $N_J$ is $g_J$-orthogonal to $dN_J(u)^{2,0}$ and $dN_J(u)^{1,1}.$ This is a consequence of the $g_J$-orthogonality of the holomorphic, and antiholomorphic tangent bundles of $X.$ \emph{From now on, all $O(u^2)$-terms will be omitted throughout with only a few exceptions in the last section.} Then, we find that   

\begin{equation*}
\begin{split}
d_N(\|N_J\|^2_{g_{J+u}})\cdot d_J N_J(u)&=\big\langle d_J N_J(u),N_J\big\rangle_{g_{J+u}}+\overline{\big\langle d_J N_J(u),N_J\big\rangle_{g_{J+u}}}\\
&=2\Re\big[\langle dN^{0,2}_J(u),N_J \rangle_{g_{J+u}} \big].
\end{split}
\end{equation*}

\begin{proposition}\label{first}
The first variation of $\mathcal{N}$ is 
\[d_J \mathcal{N}(J)(u)=\int_X \Big\{2\Re\big[\langle dN^{0,2}_J(u),N_J \rangle_{g_{J+u}}\big] + d_{g_J}(\parallel N_J \parallel^2_{g_J})\cdot \gamma\Big\} \vol_g,\] and that of $\widetilde{\mathcal{N}}$ is
\begin{equation*}
\begin{split}
d_J \widetilde{\mathcal{N}}(J)(u)&=\int_X 2\Re\big[\langle dN^{0,2}_J(u),N_J \rangle_{g_{J+u}} \big] \vol_{g_{J+u}} + \int_X d_{g_J}(\parallel N_J \parallel^2_{g_{J+u}} \vol_{g_{J+u}})\cdot \gamma+\\
&\int_X d_{\vol}(\|N_J\|^2_{g_J} \vol_{g_J})\cdot d_J(\vol_{g_J})(u).
\end{split}
\end{equation*}
\end{proposition}

In order to retrieve the Euler-Lagrange equations of interest, we need to integrate $2\Re\big[\langle dN^{0,2}_J(u),N_J \rangle_{g_{J+u}} \big]$ by parts. We do this in the coordinates defined below.

\textbf{Acknowledgment} I thank Jean-Pierre Demailly for his suggestions, and I thank the European Research Council for financial support from the grant project ``Algebraic and K{\"a}hler geometry" (ALKAGE, no.\ 670846).

\addcontentsline{toc}{section}{Coordinates}
\section*{Coordinates}
We try to develop intrinsic complex coordinates on the almost hermitian manifold $(X,J,g_J),$ centered at a given point $p \in X,$ that are the next best alternative to both holomorphic coordinates, which exist only when $J$ is integrable, and geodesic coordinates at $p,$ which exits iff the fundamental form $\omega$ of $g_J$ is K{\"a}hler. 

Recall that if $(z_k)_{1 \leq k \leq n}$ are holomorphic coordinates on $U \subset X,$ then $\bar{\partial}_J z_k=0$ on some neighborhood $U_x \subset X$ of every $x \in U.$ We do not have that in the almost complex case. However, we can design complex coordinates, which will be denoted here by $w_k,$ for which $\bar{\partial}_J w_k$ is as close as possible to being zero on each $U_x.$ First note that we can always find complex coordinates $z_k \in C^{\infty}(U_p,\mathbb{C}),$ centered at $p,$ such that $\bar{\partial}_J z_k(p)=0.$ Then, $\big(dz_k(p)\big)_{1 \leq k \leq n}=\big(\partial_J z_k(p)\big)_{1 \leq k \leq n}$ is a basis of $(T^{1,0}_{X,p})^*,$ and so $\big(\overline{dz_k(p)}\big)_{1 \leq k \leq n}=\big(\overline{\partial_J z_k(p)}\big)_{1 \leq k \leq n}$ is a basis of $(T^{0,1}_{X,p})^*.$ Hence, we get a local frame $\big(\overline{\partial_J z_k}\big)_{1 \leq k \leq n}$ of $(T^{0,1}_{X,p})^*,$ and so \[\bar{\partial}_J z_k=\sum_{1 \leq l \leq n} f_{kl}(z)\overline{\partial_J z_l},\] where $f_{kl} \in C^{\infty}(U_p,\mathbb{C}),$ and $f_{kl}(p)=0.$ Note that if $z_k$ were holomorphic, all of the coefficient functions $f_{kl}$ would be identically zero. Here, for every $1 \leq l \leq n,$ each of these functions has a Taylor expansion
\begin{equation*}
\begin{split}
f_{kl}(z)&=\sum_{|\alpha|+|\beta| \leq N} c_{kl\alpha \beta}z^{\alpha}\bar{z}^{\beta}+O(|z|^{N+1}).
\end{split}
\end{equation*}

Given that we will only differentiate once, we may instead work with the truncation 
\begin{equation*}
\begin{split}
f_{kl}(z)&=\sum_{1 \leq j \leq n} (a_{jkl}z_j + a'_{jkl} \bar{z}_j) + O(|z|^2).
\end{split}
\end{equation*}  
Again, if the $z_k$ were holomorphic, we would in particular have that $a_{jkl}=a'_{jkl}=0,$ for all $1 \leq j,k,l \leq n.$ We wish to emulate this situation ($a_{jkl}=a'_{jkl}=0$) in the almost complex case. Concretely, we are looking for new coordinates that anihilate as many of the coefficients $a_{jkl},$ and $a'_{jkl}$ as possible. To that end, let \[w_k=z_k+\sum_{1 \leq r,s \leq n} (\alpha_{krs} z_r z_s + \beta_{krs} z_r \bar{z}_s + \gamma_{krs} \bar{z}_r \bar{z}_s)+O(|z|^3).\]
We still have that $w_k(p)=0,$ $\bar{\partial}_J w_k (p)=0,$ and $dw_k(p)=dz_k(p),$ and we still get a local frame $\big(\overline{\partial_J w_k}\big)_{1 \leq k \leq n}$ of $(T^{0,1}_X|_{U_p})^*$ so that 
\begin{equation*}
\begin{split}
\bar{\partial}_J w_k&=\sum_{1 \leq j,l \leq n} (b_{jkl}w_j + b'_{jkl} \bar{w}_j+O(|w|^2))\overline{\partial_J w_l}.
\end{split}
\end{equation*}
We will see that the holomorphic condition prescribes $\beta_{klm},$ and $\gamma_{klm},$ while the geodesic condition can be used to solve for $\alpha_{klm}.$ The point is that we are reducing the problem of finding
optimal complex coordinates on an almost hermitian manifold to finding $\alpha_{krs},\beta_{krs},\gamma_{krs}$ that anihilate the maximum number of coefficients of the Taylor expansions of $f_{kl},$ and $\omega_{\lambda\bar{\mu}}.$ We call $w_k=z_k+\sum_{1 \leq r,s \leq n} (\alpha_{krs} z_r z_s + \beta_{krs} z_r \bar{z}_s + \gamma_{krs} \bar{z}_r \bar{z}_s)+O(|z|^3),$ $1 \leq k \leq n,$ with $\alpha_{krs},\beta_{krs},\gamma_{krs}$ subject to these constraints \emph{almost holomorphic geodesic coordinates on} $X$ at $p.$

\begin{lemma}\label{ach}
Any complex coordinates $z_k \in C^{\infty}(U_p,\mathbb{C}),$ $1 \leq k \leq n,$ on an almost hermitian manifold $(X,J,g_J)$ that are centered at $p,$ and for which $\bar{\partial}_J z_k(p)=0,$ and $\Big(\frac{\partial}{\partial z_k}(p)\Big)_{1 \leq k \leq n}$ is an orthonormal basis of $T^{1,0}_{X,p},$ determine almost holomorphic geodesic coordinates at $p.$ Specifically, if the Taylor expansion of $\bar{\partial}_J z_k$ on $U_p$ is \[\bar{\partial}_J z_k=\sum_{1 \leq j,l \leq n} \big(a_{kjl}z_j + a'_{kjl} \bar{z}_j +O(|z|^2)\big)\overline{\partial_J z_l},\] and if \[\omega_{m\bar{l}}=\delta_{ml}+\sum_{s=1}^n (\tau_{m\bar{l}s} z_s +\tau'_{m\bar{l}\bar{s}} \bar{z}_s)+O(|z|^2),\] then 
\begin{equation*}
\begin{split}
w_k&=z_k-\sum_{1 \leq m,l \leq n} \big[\frac{1}{4}(a_{lkm}+a_{mkl}+\tau_{l\bar{k}m}+\tau_{m\bar{k}l}) z_lz_m+a_{klm}z_l \bar{z}_m + \frac{1}{4}(a'_{klm}+a'_{kml})\bar{z}_l \bar{z}_m\big]+\\
&O(|z|^3).
\end{split}
\end{equation*}
\end{lemma}

\begin{proof}
Since $z_m \bar{\partial}_J z_l=O(|z|^2),$ $\bar{z}_m \bar{\partial}_J z_l=O(|z|^2),$ and $\bar{\partial}_J \bar{z}_l=\overline{\partial_J z_l},$ and since $\gamma_{klm}$ is $(l,m)$-symmetric, 
\begin{equation*}
\begin{split}
\bar{\partial}_J w_k &= \bar{\partial}_J z_k +\sum_{1 \leq l,m \leq n} \bar{\partial}_J \big(\alpha_{klm}z_l z_m + \beta_{klm} z_l \bar{z}_m + \gamma_{klm} \bar{z}_l \bar{z}_m)+O(|z|^3)\big)\\
&=\bar{\partial}_J z_k + \sum_{1 \leq l,m \leq n} \big(\beta_{klm} z_l \overline{\partial_J z_m}+(\gamma_{klm}+\gamma_{kml})\bar{z}_l\overline{\partial_J z_m} \big)+O(|z|^2)\\
&=\sum_{1 \leq l,m \leq n} \big(a_{klm}z_l + a'_{klm}\bar{z}_l\big)\overline{\partial_J z_m}+\sum_{1 \leq l,m \leq n} \big(\beta_{klm} z_l \overline{\partial_J z_m}+2\gamma_{klm}\bar{z}_l\overline{\partial_J z_m} \big)+O(|z|^2)\\
&=\sum_{1 \leq l,m \leq n} \big[(a_{klm}+\beta_{klm})z_l + \big(a'_{klm}+2\gamma_{klm}\big)\bar{z}_l\big]\overline{\partial_J z_m}+O(|z|^2).
\end{split}
\end{equation*}
Based on this calculation, $\alpha_{klm}$ is free to be any complex number, while $\beta_{klm}=-a_{klm}.$ And we may take, at best, the symmetric part of $a'_{klm}+2\gamma_{klm}$ to be zero, which is achieved by setting $\gamma_{klm}=-\frac{1}{4}(a'_{klm}+a'_{kml}).$ 
So far, we gathet that \[w_k=z_k+\sum_{1 \leq m,l \leq n} \big[\alpha_{klm}z_lz_m-a_{klm}z_l \bar{z}_m - \frac{1}{4}(a'_{klm}+a'_{kml})\bar{z}_l \bar{z}_m\big]+O(|z|^3).\] Next, we optimize $\alpha_{klm}$ subject to the constraint of $w_k$ being geodesic coordinates at $p.$ Since $\alpha_{mlj}$ is $(l,j)$-symmetric,

\[\partial_J w_m = \partial_J z_m + \sum_{l,j=1}^n (2 \alpha_{mlj}z_j+\beta_{mlj}\bar{z}_j)\partial_J z_l +O(|z|^2)\] so that \[\overline{\partial_J w_m} = \overline{\partial_J z_m} + \sum_{l,j=1}^n (2 \overline{\alpha_{mlj}}\bar{z}_j+\overline{\beta_{mlj}}z_j)\overline{\partial_J z_l} +O(|z|^2).\] Then, since $O(|w|^2)=O(|z|^2),$
\begin{equation*}
\begin{split}
\frac{i}{2} \sum_{m,l=1}^n \big(\delta_{ml}+O(|w|^2)\big)\partial_J w_m \wedge \overline{\partial_J w_l} &=\frac{i}{2} \sum_{m=1}^n \partial_J w_m \wedge \overline{\partial_J w_m} +O(|w|^2) \\
&=\frac{i}{2} \sum_{m=1}^n \Big[\partial_J z_m \wedge \overline{\partial_J z_m} +\\
&\sum_{l,j=1}^n (2\overline{\alpha_{mlj}} \bar{z}_j + \overline{\beta_{mlj}}z_j)\partial_J z_m \wedge \overline{\partial_J z_l}+\\
&\sum_{l,j=1}^n (2\alpha_{lmj} z_j+\beta_{lmj} \bar{z}_j)\partial_J z_m \wedge \overline{\partial_J z_l} \Big]+O(|z|^2)+O(|w|^2)\\
&=\frac{i}{2}\sum_{l,m=1}^n \Big(\delta_{ml}+\sum_{j=1}^n \big[(\overline{\beta_{mlj}}+2\alpha_{lmj})z_j +\\
&2\overline{\alpha_{mlj}} + \beta_{lmj})\bar{z}_j\big]+O(|z|^2)\Big)\partial_J z_m \wedge \overline{\partial_J z_l}\\
&=\frac{i}{2}\sum_{l,m=1}^n \Big(\delta_{ml}+\sum_{j=1}^n (\tau_{m\bar{l}j} z_j +\tau'_{m\bar{l}\bar{j}} \bar{z}_j)+O(|z|^2)\Big)\partial_J z_m \wedge \overline{\partial_J z_l}\\
&=\omega,
\end{split}
\end{equation*}
i.e.\ $\omega=\frac{i}{2} \sum_{m,l=1}^n \big(\delta_{ml}+O(|w|^2)\big)\partial_J w_m \wedge \overline{\partial_J w_l},$ is a condition that can be attained, at best, by setting the $(m,j)$-symmetric part of $\alpha_{lmj}+\frac{1}{2}(\overline{\beta_{mlj}}-\tau_{m\bar{l}j})$ equal to zero. Thus, we may take 
\begin{equation*}
\begin{split}
\alpha_{lmj}&=-\frac{1}{4}(\overline{\beta_{mlj}}+\overline{\beta_{jlm}}-\tau_{m \bar{l}j}-\tau_{j \bar{l}m})\\
&=\frac{1}{4}(\overline{a_{mlj}}+\overline{a_{jlm}}+\tau_{m \bar{l}j}+\tau_{j \bar{l}m}),
\end{split}
\end{equation*}
and therefore
\begin{equation*}
\begin{split}
w_k&=z_k+\sum_{1 \leq m,l \leq n} \big[\frac{1}{4}(\overline{a_{lkm}}+\overline{a_{mkl}}+\tau_{l\bar{k}m}+\tau_{m\bar{k}l}) z_lz_m-a_{klm}z_l \bar{z}_m - \frac{1}{4}(a'_{klm}+a'_{kml})\bar{z}_l \bar{z}_m\big]+\\
&O(|z|^3).
\end{split}
\end{equation*}
\end{proof}

\addcontentsline{toc}{section}{Euler-Lagrange equations}
\section*{Euler-Lagrange equations}
Let $(w_k)_{k=1}^n$ be almost holomorphic geodesic coordinates at $p.$ We now have local coordinate frames $\Big(\frac{\partial^{1,0}}{\partial w_k}\Big)_{1 \leq k \leq n}$ of $T^{1,0}_X,$ and $\Big(\frac{\partial^{0,1}}{\partial \bar{w}_k}\Big)_{1 \leq k \leq n}$ of $T^{0,1}_X$ with dual coframes $\big(dw_k^{1,0}\big)_{1 \leq k \leq n},$ and $\big(d\bar{w}_k^{0,1}\big)_{1\leq k \leq n}.$ We also have the local coordinate expressions 
\[(g_J)_{\lambda \bar{\mu}}=\Big\langle \frac{\partial^{1,0}}{\partial w_{\lambda}},\frac{\partial^{0,1}}{\partial \bar{w}_{\mu}}\Big\rangle_{g_J}=\delta_{\lambda \mu}+\sum_{m=1}^n(\tau_{\lambda \bar{\mu} m} w_m + \tau'_{\lambda \bar{\mu} m} \bar{w}_m)+O(|w|^2),\]  

\[N_J=\sum_{i,j,k=1}^n N^k_{\bar{i}\bar{j}} d\bar{w}_i^{0,1} \wedge d\bar{w}_j^{0,1} \otimes \frac{\partial^{1,0}}{\partial w_k},\] and likewise

\[dN^{0,2}_J(u)=\sum_{i,j,k=1}^n \big(dN_J(u)\big)^k_{\bar{i}\bar{j}} d\bar{w}_i^{0,1} \wedge d\bar{w}_j^{0,1} \otimes \frac{\partial^{1,0}}{\partial w_k}.\] Here we write $dV=\big(\frac{i}{2}\big)^n dw_1^{1,0} \wedge d\bar{w}_1^{0,1} \wedge \dots \wedge dw_n^{1,0} \wedge d\bar{w}_n^{0,1},$ $h:=g_{J+u}=g_{J}+\gamma+O(u^2),$ and \[h_{i \bar{j}}=\delta_{ij}+\sum_{m=1}^n (\tau_{i\bar{j}m} w_m + \tau'_{i\bar{j}m} \bar{w}_m)+O(|w|^2)+\gamma_{i\bar{j}}+O(u^2),\] the components of $\overline{h^{-1}}$ then being 
\begin{equation*}
\begin{split}
h^{i\bar{j}}&:=\delta_{ij}- \sum_{m=1}^n (\overline{\tau_{i\bar{j}m}} \bar{w}_m + \overline{\tau'_{i\bar{j}m}} w_m)-\overline{\gamma_{i\bar{j}}}+\sum_{c,v=1}^n (\overline{\tau_{i\bar{c}v}} \bar{w}_v + \overline{\tau'_{i\bar{c}v}} w_v)\overline{\gamma_{c \bar{j}}}+\\
&\sum_{c,v=1}^n \overline{\gamma_{i\bar{c}}}(\overline{\tau_{c\bar{j}v}} \bar{w}_v + \overline{\tau'_{c\bar{j}v}} w_v)+O(|w|^2)+O(u^2).
\end{split}
\end{equation*}

\begin{lemma}\label{deri}
\quad \\
\[\frac{\partial^{1,0} h_{r\bar{s}}}{\partial w_i}(p)=\tau_{r\bar{s}i} +O(u), \quad \frac{\partial^{0,1} h_{r\bar{s}}}{\partial \bar{w}_i}(p)=\tau'_{r\bar{s}i} +O(u),\] \[\frac{\partial^{1,0} h^{r\bar{s}}}{\partial w_i}(p)=-\overline{\tau'_{r\bar{s}i}} +O(u), \mbox{ and } \frac{\partial^{0,1} h^{r\bar{s}}}{\partial \bar{w}_i}(p)=-\overline{\tau_{r\bar{s}i}}+O(u).\]
\end{lemma}

\begin{proof}
These equalities are a consequence of $\gamma_{i\bar{j}},$ and hence its derivatives with respect to $w_m$ and $\bar{w}_m,$ being of order $O(u).$ 
\end{proof}

\begin{lemma}\label{step}
\begin{equation*}
\begin{split}
\Big\langle dN^{0,2}_J(u),N_J \Big\rangle_{g_{J+u}} &=2\Big(h^{s\bar{m}}h^{j\bar{p}}h_{k\bar{q}} J^i_{\bar{j}} \frac{\partial ^{1,0} u^k_{\bar{s}}}{\partial w_i}+h^{i\bar{m}}h^{j\bar{p}}h_{k\bar{q}}\Big(\frac{\partial ^{0,1} u^s_{\bar{j}}}{\partial \bar{w}_i}J^k_s +\frac{\partial ^{0,1} u^{\bar{s}}_{\bar{j}}}{\partial \bar{w}_i}J^k_{\bar{s}}\Big)+\\
& h^{s\bar{m}}h^{j\bar{p}}h_{k\bar{q}} \frac{\partial ^{0,1} u^k_{\bar{s}}}{\partial \bar{w}_i}J^{\bar{i}}_{\bar{j}}\Big)\overline{N^q_{\bar{m}\bar{p}}}-2 h^{s\bar{m}}h^{j\bar{p}}h_{k\bar{q}} \Big(u^i_{\bar{s}} \frac{\partial ^{1,0} J^k_{\bar{j}}}{\partial w_i}+u^{\bar{i}}_{\bar{s}} \frac{\partial ^{0,1} J^k_{\bar{j}}}{\partial \bar{w}_i}\Big)\overline{N^q_{\bar{m}\bar{p}}}+\\
&2 h^{i\bar{m}}h^{j\bar{p}}h_{k\bar{q}}\Big\langle u\Big[\frac{\partial^{0,1}}{\partial \bar{w}_i},J\frac{\partial^{0,1}}{\partial \bar{w}_j}\Big],dw_k^{1,0}\Big\rangle \overline{N^q_{\bar{m}\bar{p}}}.
\end{split}
\end{equation*}
\end{lemma}

\begin{proof}
First, note that 
\[dN_J(u)(\zeta,\eta)=u[\zeta,J\eta]+J[\zeta,u\eta]+u[J\zeta,\eta]+J[u\zeta,\eta]-[u\zeta,J\eta]-[J\zeta,u\eta].\]

Since $u=u^s_r dw_r^{1,0} \otimes \frac{\partial^{1,0}}{\partial w_s}+u^{\bar{s}}_r dw_r^{1,0} \otimes \frac{\partial^{0,1}}{\partial \bar{w}_s}+u^s_{\bar{r}} d\bar{w}_r^{0,1} \otimes \frac{\partial^{1,0}}{\partial w_s}+u^{\bar{s}}_{\bar{r}} d\bar{w}_r^{0,1} \otimes \frac{\partial^{0,1}}{\partial \bar{w}_s},$ and $J=J^v_t dw_t^{1,0} \otimes \frac{\partial^{1,0}}{\partial w_v}+u^{\bar{v}}_t dw_t^{1,0} \otimes \frac{\partial^{0,1}}{\partial \bar{w}_v}+u^v_{\bar{t}} d\bar{w}_t^{0,1} \otimes \frac{\partial^{1,0}}{\partial w_v}+u^{\bar{v}}_{\bar{t}} d\bar{w}_t^{0,1} \otimes \frac{\partial^{0,1}}{\partial \bar{w}_v},$ it follows that
\begin{equation*}
\begin{split}
J\Big[\frac{\partial^{0,1}}{\partial \bar{w}_i},u\frac{\partial^{0,1}}{\partial \bar{w}_j}\Big]&=\Big(\frac{\partial^{0,1} u^s_{\bar{j}}}{\partial \bar{w}_i} J^v_s+\frac{\partial^{0,1} u^{\bar{s}}_{\bar{j}}}{\partial \bar{w}_i} J^v_{\bar{s}}\Big)\frac{\partial^{1,0}}{\partial w_v}+\Big(\frac{\partial^{0,1} u^s_{\bar{j}}}{\partial \bar{w}_i} J^{\bar{v}}_s+\frac{\partial^{0,1} u^{\bar{s}}_{\bar{j}}}{\partial \bar{w}_i} J^{\bar{v}}_{\bar{s}}\Big)\frac{\partial^{0,1}}{\partial \bar{w}_v},
\end{split}
\end{equation*}
and
\begin{equation*}
\begin{split}
u\Big[\frac{\partial^{0,1}}{\partial \bar{w}_i},J\frac{\partial^{0,1}}{\partial \bar{w}_j}\Big]&=\Big(\frac{\partial^{0,1} J^s_{\bar{j}}}{\partial \bar{w}_i} u^v_s+\frac{\partial^{0,1} J^{\bar{s}}_{\bar{j}}}{\partial \bar{w}_i} u^v_{\bar{s}}\Big)\frac{\partial^{1,0}}{\partial w_v}+\Big(\frac{\partial^{0,1} J^s_{\bar{j}}}{\partial \bar{w}_i} u^{\bar{v}}_s+\frac{\partial^{0,1} J^{\bar{s}}_{\bar{j}}}{\partial \bar{w}_i} u^{\bar{v}}_{\bar{s}}\Big)\frac{\partial^{0,1}}{\partial \bar{w}_v}.
\end{split}
\end{equation*}

A similar computation shows that 
\begin{equation*}
\begin{split}
\Big[J \frac{\partial^{0,1}}{\partial \bar{w}_j},u\frac{\partial^{0,1}}{\partial \bar{w}_i}\Big]&=\Big(J^v_{\bar{j}} \frac{\partial^{1,0} u^s_{\bar{i}}}{\partial w_v}+J^{\bar{v}}_{\bar{j}}\frac{\partial^{0,1} u^s_{\bar{i}}}{\partial \bar{w}_v}-u^v_{\bar{i}} \frac{\partial^{1,0} J^s_{\bar{j}}}{\partial w_v}-u^{\bar{v}}_{\bar{i}} \frac{\partial^{0,1} J^s_{\bar{j}}}{\partial \bar{w}_v}\Big)\frac{\partial^{1,0}}{\partial w_s}+\\
&\Big(J^v_{\bar{j}} \frac{\partial^{1,0} u^{\bar{s}}_{\bar{i}}}{\partial w_v}+J^v_{\bar{j}}\frac{\partial^{0,1} u^{\bar{s}}_{\bar{i}}}{\partial \bar{w}_v}-u^v_{\bar{i}} \frac{\partial^{1,0} J^{\bar{s}}_{\bar{j}}}{\partial w_v}-u^{\bar{v}}_{\bar{i}} \frac{\partial^{0,1} J^{\bar{s}}_{\bar{j}}}{\partial \bar{w}_v}\Big)\frac{\partial^{0,1}}{\partial \bar{w}_s}.
\end{split}
\end{equation*}

Therefore,
\begin{equation*}
\begin{split}
\Big\langle J\Big[\frac{\partial^{0,1}}{\partial \bar{w}_i},u\frac{\partial^{0,1}}{\partial \bar{w}_j}\Big]+\Big[J \frac{\partial^{0,1}}{\partial \bar{w}_j},u\frac{\partial^{0,1}}{\partial \bar{w}_i}\Big],dw_k^{1,0}\Big\rangle&=\Big(\frac{\partial^{0,1} u^s_{\bar{j}}}{\partial \bar{w}_i} J^k_s+\frac{\partial^{0,1} u^{\bar{s}}_{\bar{j}}}{\partial \bar{w}_i} J^k_{\bar{s}}\Big)+\\
&\Big(J^s_{\bar{j}} \frac{\partial^{1,0} u^k_{\bar{i}}}{\partial w_s}+J^{\bar{s}}_{\bar{j}}\frac{\partial^{0,1} u^k_{\bar{i}}}{\partial \bar{w}_s}-u^s_{\bar{i}} \frac{\partial^{1,0} J^k_{\bar{j}}}{\partial w_s}-u^{\bar{s}}_{\bar{i}} \frac{\partial^{0,1} J^k_{\bar{j}}}{\partial \bar{w}_s}\Big),
\end{split}
\end{equation*}

and so 
\begin{align*}
\begin{split}
\Big\langle dN^{0,2}_J(u),N_J \Big\rangle_{g_{J+u}} &= h^{i\bar{m}}h^{j\bar{p}}h_{k\bar{q}} \big(dN_J(u)\big)^k_{\bar{i}\bar{j}}  \overline{N^q_{\bar{m}\bar{p}}}\\
&= h^{i\bar{m}}h^{j\bar{p}}h_{k\bar{q}}\Big\langle u\Big[ \frac{\partial^{0,1}}{\partial \bar{w}_i},J\frac{\partial^{0,1}}{\partial \bar{w}_j}\Big]+J\Big[\frac{\partial^{0,1}}{\partial \bar{w}_i},u\frac{\partial^{0,1}}{\partial \bar{w}_j}\Big]+u\Big[J \frac{\partial^{0,1}}{\partial \bar{w}_i},\frac{\partial^{0,1}}{\partial \bar{w}_j}\Big]+\\
&J\Big[u\frac{\partial^{0,1}}{\partial \bar{w}_i},\frac{\partial^{0,1}}{\partial \bar{w}_j}\Big]-\Big[u\frac{\partial^{0,1}}{\partial \bar{w}_i},J\frac{\partial^{0,1}}{\partial \bar{w}_j}\Big]-\Big[J \frac{\partial^{0,1}}{\partial \bar{w}_i},u\frac{\partial^{0,1}}{\partial \bar{w}_j}\Big],dw_k^{1,0}\Big\rangle \overline{N^q_{\bar{m}\bar{p}}}\\
&= 2h^{i\bar{m}}h^{j\bar{p}}h_{k\bar{q}}\Big\langle J\Big[\frac{\partial^{0,1}}{\partial \bar{w}_i},u\frac{\partial^{0,1}}{\partial \bar{w}_j}\Big]+\Big[J \frac{\partial^{0,1}}{\partial \bar{w}_j},u\frac{\partial^{0,1}}{\partial \bar{w}_i}\Big],dw_k^{1,0}\Big\rangle \overline{N^q_{\bar{m}\bar{p}}}+\\
&2 h^{i\bar{m}}h^{\bar{j}p}h_{k\bar{q}}\Big\langle u\Big[\frac{\partial^{0,1}}{\partial \bar{w}_i},J\frac{\partial^{0,1}}{\partial \bar{w}_j}\Big],dw_k^{1,0}\Big\rangle \overline{N^q_{\bar{m}\bar{p}}}\\
&=2\Big(h^{s\bar{m}}h^{j\bar{p}}h_{k\bar{q}} J^i_{\bar{j}} \frac{\partial ^{1,0} u^k_{\bar{s}}}{\partial w_i}+h^{i\bar{m}}h^{j\bar{p}}h_{k\bar{q}}\Big(\frac{\partial ^{0,1} u^s_{\bar{j}}}{\partial \bar{w}_i}J^k_s +\frac{\partial ^{0,1} u^{\bar{s}}_{\bar{j}}}{\partial \bar{w}_i}J^k_{\bar{s}}\Big)+\\
& h^{s\bar{m}}h^{j\bar{p}}h_{k\bar{q}} \frac{\partial ^{0,1} u^k_{\bar{s}}}{\partial \bar{w}_i}J^{\bar{i}}_{\bar{j}}\Big)\overline{N^q_{\bar{m}\bar{p}}}-2 h^{s\bar{m}}h^{j\bar{p}}h_{k\bar{q}} \Big(u^i_{\bar{s}} \frac{\partial ^{1,0} J^k_{\bar{j}}}{\partial w_i}+u^{\bar{i}}_{\bar{s}} \frac{\partial ^{0,1} J^k_{\bar{j}}}{\partial \bar{w}_i}\Big)\overline{N^q_{\bar{m}\bar{p}}}+\\
&2 h^{i\bar{m}}h^{j\bar{p}}h_{k\bar{q}}\Big\langle u\Big[\frac{\partial^{0,1}}{\partial \bar{w}_i},J\frac{\partial^{0,1}}{\partial \bar{w}_j}\Big],dw_k^{1,0}\Big\rangle \overline{N^q_{\bar{m}\bar{p}}}.
\end{split}
\end{align*}
\end{proof}

\begin{lemma}\label{libp}
At $p,$ we have that
\begin{equation*}
\begin{split}
\frac{\partial^{1,0}}{\partial w_i}\Big[h^{s\bar{m}} h^{j\bar{p}} h_{k\bar{q}} J^i_{\bar{j}} \overline{N^q_{\bar{m}\bar{p}}} \det{(h)} \Big]u_{\bar{s}}^k &=\Big(-\overline{\tau'_{s\bar{m}i}} J^i_{\bar{j}} \overline{N^k_{\bar{m}\bar{j}}}-\overline{\tau'_{j\bar{p}i}} J^i_{\bar{j}} \overline{N^k_{\bar{s}\bar{p}}}+\tau_{k\bar{q}i} J^i_{\bar{j}} \overline{N^q_{\bar{s}\bar{j}}}+\\
&\frac{\partial^{1,0} J^i_{\bar{j}}}{\partial w_i} \overline{N^k_{\bar{s}\bar{j}}}+J^i_{\bar{j}} \frac{\partial^{1,0} \overline{N^k_{\bar{s}\bar{j}}}}{\partial w_i}+J^i_{\bar{j}} \overline{N^k_{\bar{s}\bar{j}}} \big(\sum_{c=1}^n \tau_{c\bar{c}i}\big)\Big) u^k_{\bar{s}},
\end{split}
\end{equation*}

\begin{equation*}
\begin{split}
\frac{\partial^{0,1}}{\partial \bar{w}_i}\Big[h^{i\bar{m}} h^{j\bar{p}} h_{k\bar{q}} J^k_s \overline{N^q_{\bar{m}\bar{p}}} \det{(h)} \Big]u_{\bar{j}}^s &=\Big(-\overline{\tau_{i\bar{m}i}} J^k_s \overline{N^k_{\bar{m}\bar{j}}}-\overline{\tau_{j\bar{p}i}} J^k_s \overline{N^k_{\bar{i}\bar{p}}}+\tau'_{k\bar{q}i} J^k_s \overline{N^q_{\bar{i}\bar{j}}}+\\
&\frac{\partial^{0,1} J^k_s}{\partial \bar{w}_i} \overline{N^k_{\bar{i}\bar{j}}}+J^k_s \frac{\partial^{0,1} \overline{N^k_{\bar{i}\bar{j}}}}{\partial \bar{w}_i}+J^k_s \overline{N^k_{\bar{i}\bar{j}}} \big(\sum_{c=1}^n \tau'_{c\bar{c}i}\big)\Big) u^s_{\bar{j}},
\end{split}
\end{equation*}

\begin{equation*}
\begin{split}
\frac{\partial^{0,1}}{\partial \bar{w}_i}\Big[h^{i\bar{m}} h^{j\bar{p}} h_{k\bar{q}} J^k_{\bar{s}} \overline{N^q_{\bar{m}\bar{p}}} \det{(h)} \Big]u_{\bar{j}}^{\bar{s}} &=\Big(-\overline{\tau_{i\bar{m}i}} J^k_{\bar{s}} \overline{N^k_{\bar{m}\bar{j}}}-\overline{\tau_{j\bar{p}i}} J^k_{\bar{s}} \overline{N^k_{\bar{i}\bar{p}}}+\tau'_{k\bar{q}i} J^k_{\bar{s}} \overline{N^q_{\bar{i}\bar{j}}}+\\
&\frac{\partial^{0,1} J^k_{\bar{s}}}{\partial \bar{w}_i} \overline{N^k_{\bar{i}\bar{j}}}+J^k_{\bar{s}} \frac{\partial^{0,1} \overline{N^k_{\bar{i}\bar{j}}}}{\partial \bar{w}_i}+J^k_{\bar{s}} \overline{N^k_{\bar{i}\bar{j}}} \big(\sum_{c=1}^n \tau'_{c\bar{c}i}\big)\Big) u^{\bar{s}}_{\bar{j}},
\end{split}
\end{equation*}

and 

\begin{equation*}
\begin{split}
\frac{\partial^{0,1}}{\partial \bar{w}_i}\Big[h^{s\bar{m}} h^{j\bar{p}} h_{k\bar{q}} J^{\bar{i}}_{\bar{j}} \overline{N^q_{\bar{m}\bar{p}}} \det{(h)} \Big]u_{\bar{s}}^k &=\Big(-\overline{\tau_{s\bar{m}i}} J^{\bar{i}}_{\bar{j}} \overline{N^k_{\bar{m}\bar{j}}}-\overline{\tau_{j\bar{p}i}} J^{\bar{i}}_{\bar{j}} \overline{N^k_{\bar{s}\bar{p}}}+\tau'_{k\bar{q}i} J^{\bar{i}}_{\bar{j}} \overline{N^q_{\bar{s}\bar{j}}}+\\
&\frac{\partial^{0,1} J^{\bar{i}}_{\bar{j}}}{\partial \bar{w}_i} \overline{N^k_{\bar{s}\bar{j}}}+J^{\bar{i}}_{\bar{j}}\frac{\partial^{0,1} \overline{N^k_{\bar{s}\bar{j}}}}{\partial \bar{w}_i}+J^{\bar{i}}_{\bar{j}} \overline{N^k_{\bar{s}\bar{j}}} \big(\sum_{c=1}^n \tau'_{c\bar{c}i}\big)\Big) u^k_{\bar{s}}.
\end{split}
\end{equation*}
\end{lemma}

\begin{proof}
Note that $h_{i\bar{j}}(p)=\delta_{ij}+\gamma_{i\bar{j}}(p)=\delta_{ij}+O(u),$ and $h^{i\bar{j}} (p)=\delta_{ij}-\overline{\gamma_{i\bar{j}}}=\delta_{ij}+O(u).$ Then using Lemma \ref{deri}, \[\frac{\partial^{1,0} h^{s\bar{m}}}{\partial w_i} h^{j\bar{p}} h_{k\bar{q}}=-\overline{\tau'_{s\bar{m}i}} \delta_{jp} \delta_{kq} +O(u), \quad h^{s\bar{m}} \frac{\partial^{1,0} h^{j\bar{p}}}{\partial w_i} h_{k\bar{q}}=-\overline{\tau'_{j\bar{p}i}} \delta_{sm} \delta_{kq} +O(u),\] and \[h^{s\bar{m}} h^{j\bar{p}} \frac{\partial^{1,0} h_{k\bar{q}}}{\partial w_i}=\tau_{k\bar{q}i} \delta_{sm} \delta_{jp} +O(u).\] Now since $\det{h(p)}=1+\sum_{c=1}^n \gamma_{c\bar{c}}=1+O(u),$ 
\begin{equation}\label{op}
\begin{split}
\frac{\partial^{1,0} h^{s\bar{m}} h^{j\bar{p}} h_{k\bar{q}}}{\partial w_i} J^i_{\bar{j}} \overline{N^q_{\bar{m}\bar{p}}} \det{(h)} u^k_{\bar{s}}&=\Big(\frac{\partial^{1,0} h^{s\bar{m}}}{\partial w_i} h^{j\bar{p}} h_{k\bar{q}} + h^{s\bar{m}} \frac{\partial^{1,0} h^{j\bar{p}}}{\partial w_i}h_{k\bar{q}} +\\
&h^{s\bar{m}} h^{j\bar{p}} \frac{\partial^{1,0} h_{k\bar{q}}}{\partial w_i}  \Big)J^i_{\bar{j}} \overline{N^q_{\bar{m}\bar{p}}} \det{(h)} u^k_{\bar{s}}\\
&=\big(-\overline{\tau'_{s\bar{m}i}} \delta_{jp} \delta_{kq}-\overline{\tau'_{j\bar{p}i}} \delta_{sm} \delta_{kq}+\tau_{k\bar{q}i} \delta_{sm} \delta_{jp}\big)J^i_{\bar{j}} \overline{N^q_{\bar{m}\bar{p}}} u^k_{\bar{s}}.
\end{split}
\end{equation}
Moreover, $h^{s\bar{m}} h^{j\bar{p}} h_{k\bar{q}} (p)= \delta_{sm}\delta_{jp} \delta_{kq} +O(u),$ and since \[\det{(h)}=1+\sum_{c,m=1}^n (\tau_{c\bar{c}m} w_m + \tau'_{c\bar{c}m} \bar{w}_m) + \sum_{c=1}^n \gamma_{c\bar{c}} +O(|w|^2),\] \[\frac{\partial^{1,0} \det{(h)}}{\partial w_i} (p)=\sum_{c=1}^n \tau_{c\bar{c}i} +\sum_{c=1}^n \frac{\partial^{1,0} \gamma_{c\bar{c}}}{w_i}(p)=\sum_{c=1}^n \tau_{c\bar{c}i} +O(u),\] implying that
\begin{equation}\label{opp}
\begin{split}
h^{s\bar{m}} h^{j\bar{p}} h_{k\bar{q}} \frac{\partial^{1,0}}{\partial w_i} \Big(J^i_{\bar{j}} \overline{N^q_{\bar{m}\bar{p}}} \det{(h)}\big)u^k_{\bar{s}}&=h^{s\bar{m}} h^{j\bar{p}} h_{k\bar{q}} \Big(\frac{\partial^{1,0} J^i_{\bar{j}}}{\partial w_i}\overline{N^q_{\bar{m}\bar{p}}} \det{(h)} +\\
&J^i_{\bar{j}} \frac{\partial^{1,0} \overline{N^q_{\bar{m}\bar{p}}}}{\partial w_i}\det{(h)}+J^i_{\bar{j}} \overline{N^q_{\bar{m}\bar{p}}} \frac{\partial^{1,0} \det{(h)}}{\partial w_i}\Big)u^k_{\bar{s}}\\
&=\delta_{sm} \delta_{jp} \delta_{kq} \Big(\frac{\partial^{1,0} J^i_{\bar{j}}}{\partial w_i}\overline{N^q_{\bar{m}\bar{p}}}+J^i_{\bar{j}} \frac{\partial^{1,0} \overline{N^q_{\bar{m}\bar{p}}}}{\partial w_i}+J^i_{\bar{j}} \overline{N^q_{\bar{m}\bar{p}}} \big(\sum_{c=1}^n \tau_{c\bar{c}i}\big)\Big)u^k_{\bar{s}}.
\end{split}
\end{equation}

Now, using equations \ref{op} and \ref{opp}, we see that 

\begin{equation*}
\begin{split}
\frac{\partial^{1,0}}{\partial w_i}\Big[h^{s\bar{m}} h^{j\bar{p}} h_{k\bar{q}} J^i_{\bar{j}} \overline{N^q_{\bar{m}\bar{p}}} \det{(h)} \Big]u_{\bar{s}}^k &=\Big[\big(-\overline{\tau'_{s\bar{m}i}} \delta_{jp} \delta_{kq}-\overline{\tau'_{j\bar{p}i}} \delta_{sm} \delta_{kq}+\tau_{k\bar{q}i} \delta_{sm} \delta_{jp}\big)J^i_{\bar{j}} \overline{N^q_{\bar{m}\bar{p}}}+\\
&\delta_{sm} \delta_{jp} \delta_{kq} \Big(\frac{\partial^{1,0} J^i_{\bar{j}}}{\partial w_i} \overline{N^q_{\bar{m}\bar{p}}}+J^i_{\bar{j}} \frac{\partial^{1,0} \overline{N^q_{\bar{m}\bar{p}}}}{\partial w_i}+J^i_{\bar{j}} \overline{N^q_{\bar{m}\bar{p}}} \big(\sum_{c=1}^n \tau_{c\bar{c}i}\big)\Big) \Big]u^k_{\bar{s}}\\
&=\Big(-\overline{\tau'_{s\bar{m}i}} J^i_{\bar{j}} \overline{N^k_{\bar{m}\bar{j}}}-\overline{\tau'_{j\bar{p}i}} J^i_{\bar{j}} \overline{N^k_{\bar{s}\bar{p}}}+\tau_{k\bar{q}i} J^i_{\bar{j}} \overline{N^q_{\bar{s}\bar{j}}}+\\
&\frac{\partial^{1,0} J^i_{\bar{j}}}{\partial w_i} \overline{N^k_{\bar{s}\bar{j}}}+J^i_{\bar{j}} \frac{\partial^{1,0} \overline{N^k_{\bar{s}\bar{j}}}}{\partial w_i}+J^i_{\bar{j}} \overline{N^k_{\bar{s}\bar{j}}} \big(\sum_{c=1}^n \tau_{c\bar{c}i}\big)\Big) u^k_{\bar{s}}.
\end{split}
\end{equation*}

Again, using Lemma \ref{deri}, we find that 

\begin{equation*}
\begin{split}
\frac{\partial^{0,1} h^{i\bar{m}} h^{j\bar{p}} h_{k\bar{q}}}{\partial \bar{w}_i} J^k_s \overline{N^q_{\bar{m}\bar{p}}} \det{(h)} u^s_{\bar{j}}&=\Big(\frac{\partial^{0,1} h^{i\bar{m}}}{\partial \bar{w}_i} h^{j\bar{p}} h_{k\bar{q}} + h^{i\bar{m}} \frac{\partial^{0,1} h^{j\bar{p}}}{\partial \bar{w}_i}h_{k\bar{q}} +\\
&h^{i\bar{m}} h^{j\bar{p}} \frac{\partial^{0,1} h_{k\bar{q}}}{\partial \bar{w}_i}  \Big)J^k_s \overline{N^q_{\bar{m}\bar{p}}} \det{(h)} u^s_{\bar{j}}\\
&=\big(-\overline{\tau_{i\bar{m}i}} \delta_{jp} \delta_{kq}-\overline{\tau_{j\bar{p}i}} \delta_{im} \delta_{kq}+\tau'_{k\bar{q}i} \delta_{im} \delta_{jp}\big)J^k_s \overline{N^q_{\bar{m}\bar{p}}} u^s_{\bar{j}}.
\end{split}
\end{equation*}

Since \[\frac{\partial^{0,1} \det{(h)}}{\partial \bar{w}_i} (p)=\sum_{c=1}^n \tau'_{c\bar{c}i}+O(u),\]
\begin{equation}
\begin{split}
h^{i\bar{m}} h^{j\bar{p}} h_{k\bar{q}} \frac{\partial^{0,1}}{\partial \bar{w}_i} \Big(J^k_s \overline{N^q_{\bar{m}\bar{p}}} \det{(h)}\Big)u^s_{\bar{j}}&=h^{i\bar{m}} h^{j\bar{p}} h_{k\bar{q}} \Big(\frac{\partial^{0,1} J^k_s}{\partial \bar{w}_i}\overline{N^q_{\bar{m}\bar{p}}} \det{(h)} +\\
&J^k_s \frac{\partial^{0,1} \overline{N^q_{\bar{m}\bar{p}}}}{\partial \bar{w}_i}\det{(h)}+J^k_s \overline{N^q_{\bar{m}\bar{p}}} \frac{\partial^{0,1} \det{(h)}}{\partial \bar{w}_i}\Big)u^s_{\bar{j}}\\
&=\delta_{im} \delta_{jp} \delta_{kq} \Big(\frac{\partial^{0,1} J^k_s}{\partial \bar{w}_i}\overline{N^q_{\bar{m}\bar{p}}}+J^k_s \frac{\partial^{0,1} \overline{N^q_{\bar{m}\bar{p}}}}{\partial \bar{w}_i}+J^k_s \overline{N^q_{\bar{m}\bar{p}}} \big(\sum_{c=1}^n \tau'_{c\bar{c}i}\big)\Big)u^s_{\bar{j}},
\end{split}
\end{equation}
confirming that 

\begin{equation*}
\begin{split}
\frac{\partial^{0,1}}{\partial \bar{w}_i}\Big[h^{i\bar{m}} h^{j\bar{p}} h_{k\bar{q}} J^k_s \overline{N^q_{\bar{m}\bar{p}}} \det{(h)} \Big]u_{\bar{j}}^s &=\Big[\big(-\overline{\tau_{i\bar{m}i}} \delta_{jp} \delta_{kq}-\overline{\tau_{j\bar{p}i}} \delta_{im} \delta_{kq}+\tau'_{k\bar{q}i} \delta_{im} \delta_{jp}\big)J^k_s \overline{N^q_{\bar{m}\bar{p}}}+\\
&\delta_{im} \delta_{jp} \delta_{kq} \Big(\frac{\partial^{0,1} J^k_s}{\partial \bar{w}_i} \overline{N^q_{\bar{m}\bar{p}}}+J^k_s \frac{\partial^{0,1} \overline{N^q_{\bar{m}\bar{p}}}}{\partial \bar{w}_i}+J^k_s \overline{N^q_{\bar{m}\bar{p}}} \big(\sum_{c=1}^n \tau'_{c\bar{c}i}\big)\Big) \Big]u^s_{\bar{j}}\\
&=\Big(-\overline{\tau_{i\bar{m}i}} J^k_s \overline{N^k_{\bar{m}\bar{j}}}-\overline{\tau_{j\bar{p}i}} J^k_s \overline{N^k_{\bar{i}\bar{p}}}+\tau'_{k\bar{q}i} J^k_s \overline{N^q_{\bar{i}\bar{j}}}+\\
&\frac{\partial^{0,1} J^k_s}{\partial \bar{w}_i} \overline{N^k_{\bar{i}\bar{j}}}+J^k_s \frac{\partial^{0,1} \overline{N^k_{\bar{i}\bar{j}}}}{\partial \bar{w}_i}+J^k_s \overline{N^k_{\bar{i}\bar{j}}} \big(\sum_{c=1}^n \tau'_{c\bar{c}i}\big)\Big) u^s_{\bar{j}}.
\end{split}
\end{equation*}
\end{proof}

\begin{lemma}\label{gam}
Let $g'_{ij}=\frac{1}{2}g\big(\frac{\partial^{1,0}}{\partial w_i},J\frac{\partial^{1,0}}{\partial w_j}\big),$ $g'_{\bar{i}j}=\frac{1}{2}g\big(\frac{\partial^{0,1}}{\partial \bar{w}_i},J\frac{\partial^{1,0}}{\partial w_j}\big),$ and so forth. Then,
\begin{equation}\label{eq10}
\begin{split}
d_{g_J}\big(\|N\|^2_{g_{J+u}} \vol_{g_{J+u}}\big)\cdot \gamma \big (p)&=-\Big(u^k_s \big(2g'_{k\bar{j}} N^v_{\bar{i}\bar{j}} \overline{N^v_{\bar{i}\bar{s}}}-g'_{k\bar{m}} N^s_{\bar{i}\bar{j}} \overline{N^m_{\bar{i}\bar{j}}}\big)+u^{\bar{k}}_s \big(2g'_{\bar{k}\bar{j}} N^v_{\bar{i}\bar{j}} \overline{N^v_{\bar{i}\bar{s}}}-g'_{\bar{k}\bar{m}} N^s_{\bar{i}\bar{j}} \overline{N^m_{\bar{i}\bar{j}}}\big)+\\
&u^k_{\bar{s}} \big(2g'_{km} N^v_{\bar{i}\bar{s}} \overline{N^v_{\bar{i}\bar{m}}}-g'_{kv} N^v_{\bar{i}\bar{j}} \overline{N^s_{\bar{i}\bar{j}}}\big)+u^{\bar{k}}_{\bar{s}} \big(2g'_{\bar{k}m} N^v_{\bar{i}\bar{s}} \overline{N^v_{\bar{i}\bar{m}}}-g'_{\bar{k}v} N^v_{\bar{i}\bar{j}} \overline{N^s_{\bar{i}\bar{j}}}\big)\Big)dV,
\end{split}
\end{equation}
and 
\begin{equation}\label{eq11}
\begin{split}
d_{\vol} \big(\|N_J\|^2 \vol_{g_J} \big) \cdot d_J (\vol_{g_J})(u)(p)&=\Big(\big(u^k_s g'_{k\bar{s}}+u^{\bar{k}}_s g'_{\bar{k}\bar{s}} + u^k_{\bar{s}} g'_{ks}+u^{\bar{k}}_{\bar{s}} g'_{\bar{k}s} \big)|N^v_{\bar{i}\bar{j}}|^2\Big)dV
\end{split}
\end{equation}
\end{lemma}
\begin{proof}
First note that
\begin{equation*}
\begin{split}
d_{g_J}\big(\|N\|^2_{g_{J+u}} \vol_{g_{J+u}}\big)\cdot \gamma \big (p)&=\big(\|N_J\|^2_{I+\gamma}-\|N_J\|^2_{I}\big)\big(1+tr(\gamma)\big)dV\\
&=-\big(\overline{\gamma_{j\bar{m}}} N^k_{\bar{i}\bar{j}} \overline{N^k_{\bar{i}\bar{m}}}+\overline{\gamma_{i\bar{m}}} N^k_{\bar{i}\bar{j}} \overline{N^k_{\bar{m}\bar{j}}}-\gamma_{k\bar{m}} N^k_{\bar{i}\bar{j}} \overline{N^m_{\bar{i}\bar{j}}} \big)dV\\
&=-\big(2\overline{\gamma_{j\bar{m}}} N^k_{\bar{i}\bar{j}} \overline{N^k_{\bar{i}\bar{m}}}-\gamma_{k\bar{m}} N^k_{\bar{i}\bar{j}} \overline{N^m_{\bar{i}\bar{j}}} \big)dV,
\end{split}
\end{equation*}

and that

\begin{equation*}
\begin{split}
d_{\vol} \big(\|N_J\|^2 \vol_{g_J} \big) \cdot d_J (\vol_{g_J})(u)\big (p)&=\|N_J\|^2_I \big(\vol_{I+\gamma}-\vol_I\big)\\
&=|N^k_{\bar{i}\bar{j}}|^2 tr(\gamma) dV.
\end{split}
\end{equation*}
Also, $\gamma_{k\bar{m}}=u^v_k g'_{v\bar{m}}+u^{\bar{v}}_k g'_{\bar{v}\bar{m}} + u^v_{\bar{m}} g'_{vk}+u^{\bar{v}}_{\bar{m}} g'_{\bar{v}k}.$ Equation \ref{eq10} is obtained by a relabeling of indices in $\gamma_{m\bar{j}},$ $\gamma_{m\bar{i}},$ and $\gamma_{k\bar{m}}$ so that the upper index of $u$ is $k$ (or $\bar{k}$) and the lower index is $s$ (or $\bar{s}$), and collecting terms with the same $u$-coefficients. Equation \ref{eq11} follows after writing $tr(\gamma)=\sum_{s=1}^n \gamma_{s\bar{s}}=u_s^v g'_{v\bar{s}} + u^{\bar{v}}_s g'_{\bar{v}\bar{s}} + u^v_{\bar{s}} g'_{vs} + u^{\bar{v}}_{\bar{s}} g'_{\bar{v}s},$ and a similar relabelling of indices.
\end{proof}
\begin{proposition}
Suppose that $u$ is compactly supported in $U_p.$ Let $1 \leq p,q \leq n.$ The Euler-Lagrange system of equations of $\widetilde{\mathcal{N}}$ at $p$ is 

\[\widetilde{\mathcal{T}}^q_p:=4\frac{\partial^{0,1} J^p_{\bar{j}}}{\partial \bar{w}_i}\overline{N^q_{\bar{i}\bar{j}}}-\big(2g'_{q\bar{j}} N^v_{\bar{i}\bar{j}} \overline{N^v_{\bar{i}\bar{p}}}-g'_{q\bar{m}} N^p_{\bar{i}\bar{j}} \overline{N^m_{\bar{i}\bar{j}}}\big)+g'_{q\bar{p}} |N^v_{\bar{i}\bar{j}}|^2=0,\]

\[\widetilde{\mathcal{T}}^{\bar{q}}_p:=-\big(2g'_{\bar{q}\bar{j}} N^v_{\bar{i}\bar{j}} \overline{N^v_{\bar{i}\bar{p}}}-g'_{\bar{q}\bar{m}} N^p_{\bar{i}\bar{j}} \overline{N^m_{\bar{i}\bar{j}}}\big)+g'_{\bar{q}\bar{p}} |N^v_{\bar{i}\bar{j}}|^2=0,\]

\begin{equation*}
\begin{split}
&\widetilde{\mathcal{T}}^q_{\bar{p}}:=4\Big[\Big(J\Big(\frac{\partial^{0,1}}{\partial \bar{w}_j}\Big)\omega_{m\bar{p}}\Big)\overline{N^q_{\bar{m}\bar{j}}}+\Big(J\Big(\frac{\partial^{0,1}}{\partial \bar{w}_j}\Big)\omega_{m\bar{j}}\Big)\overline{N^q_{\bar{p}\bar{m}}}-\Big(J\Big(\frac{\partial^{0,1}}{\partial \bar{w}_j}\Big)\omega_{q\bar{m}}\Big)\overline{N^m_{\bar{p}\bar{j}}}-\frac{\partial^{1,0}}{\partial w_i} \big(J^i_{\bar{j}}\overline{N^q_{\bar{p}\bar{j}}}\big)\\
&-\frac{\partial^{0,1}}{\partial \bar{w}_i}\big(J^{\bar{i}}_{\bar{j}}\overline{N^q_{\bar{p}\bar{j}}}\big)-\Big(J\Big(\frac{\partial^{0,1}}{\partial \bar{w}_j}\Big)\Big(\sum_{c=1}^n \omega_{c\bar{c}} \Big)\Big)\overline{N^q_{\bar{p}\bar{j}}}+J^j_q \Big(\frac{\partial^{0,1} \omega_{m\bar{i}}}{\partial \bar{w}_i}\overline{N^j_{\bar{m}\bar{p}}}+\frac{\partial^{0,1} \omega_{m\bar{p}}}{\partial \bar{w}_i}\overline{N^j_{\bar{i}\bar{m}}}\\
&-\frac{\partial^{0,1} \omega_{j\bar{m}}}{\partial \bar{w}_i}\overline{N^m_{\bar{i}\bar{p}}}-\sum_{c=1}^n \frac{\partial^{0,1} \omega_{c\bar{c}}}{\partial \bar{w}_i}\overline{N^j_{\bar{i}\bar{p}}}\Big)-\frac{\partial^{0,1}}{\partial \bar{w}_i}\big(J^j_q \overline{N^j_{\bar{i}\bar{p}}}\big)-\frac{\partial^{1,0} J^i_{\bar{j}}}{\partial w_q}\overline{N^i_{\bar{p}\bar{j}}}+\frac{\partial^{0,1} J^{\bar{p}}_{\bar{j}}}{\partial \bar{w}_i} \overline{N^q_{\bar{i}\bar{j}}}\Big]\\
&-\big(2g'_{qm} N^v_{\bar{i}\bar{p}} \overline{N^v_{\bar{i}\bar{m}}}-g'_{qv} N^v_{\bar{i}\bar{j}} \overline{N^p_{\bar{i}\bar{j}}}\big)+g'_{qp} |N^v_{\bar{i}\bar{j}}|^2=0,
\end{split}
\end{equation*}
\begin{equation*}
\begin{split}
&\widetilde{\mathcal{T}}^{\bar{q}}_{\bar{p}}:=4\Big[J^j_{\bar{q}} \Big(\frac{\partial^{0,1} \omega_{m\bar{i}}}{\partial \bar{w}_i} \overline{N^j_{\bar{m}\bar{p}}}+\frac{\partial^{0,1} \omega_{m\bar{p}}}{\partial \bar{w}_i} \overline{N^j_{\bar{i}\bar{m}}}-\frac{\partial^{0,1} \omega_{j\bar{m}}}{\partial \bar{w}_i} \overline{N^m_{\bar{i}\bar{p}}}-\sum_{c=1}^n \frac{\partial^{0,1} \omega_{c\bar{c}}}{\partial \bar{w}_i} \overline{N^m_{\bar{p}\bar{j}}}\Big)\\
&-\frac{\partial^{0,1}}{\partial \bar{w}_i} \big(J^j_{\bar{q}} \overline{N^j_{\bar{i}\bar{p}}}\big)-\frac{\partial^{0,1} J^{\bar{i}}_{\bar{j}}}{\partial \bar{w}_q} \overline{N^i_{\bar{p}\bar{j}}}\Big]-\big(2g'_{\bar{q}m} N^v_{\bar{i}\bar{p}} \overline{N^v_{\bar{i}\bar{m}}}-g'_{\bar{q}v} N^v_{\bar{i}\bar{j}} \overline{N^p_{\bar{i}\bar{j}}}\big)+g'_{\bar{q}p} |N^v_{\bar{i}\bar{j}}|^2=0,
\end{split}
\end{equation*}
and that of $\mathcal{N}$ is 
\[\widetilde{\mathcal{T}}^q_p-g'_{q\bar{p}} |N^v_{\bar{i}\bar{j}}|^2=0,\]

\[\widetilde{\mathcal{T}}^{\bar{q}}_p-g'_{\bar{q}\bar{p}} |N^v_{\bar{i}\bar{j}}|^2=0,\]

\[\widetilde{\mathcal{T}}^q_{\bar{p}}+4\Big[\Big(J\Big(\frac{\partial^{0,1}}{\partial \bar{w}_j}\Big)\Big(\sum_{c=1}^n \omega_{c\bar{c}} \Big)\Big)\overline{N^q_{\bar{p}\bar{j}}}+J^j_q \sum_{c=1}^n \frac{\partial^{0,1} \omega_{c\bar{c}}}{\partial \bar{w}_i}\overline{N^j_{\bar{i}\bar{p}}}  \Big]-g'_{qp} |N^v_{\bar{i}\bar{j}}|^2=0,\]

\[\widetilde{\mathcal{T}}^{\bar{q}}_{\bar{p}}+4J^j_{\bar{q}} \sum_{c=1}^n \frac{\partial^{0,1} \omega_{c\bar{c}}}{\partial \bar{w}_i} \overline{N^m_{\bar{p}\bar{j}}}-g'_{\bar{q}p} |N^v_{\bar{i}\bar{j}}|^2=0.\]
\end{proposition}

\begin{proof}
The procedure is to use Lemma \ref{libp} to integrate by parts the terms involving derivatives of $u$ in the first variation of $\widetilde{\mathcal{N}}$ (Proposition \ref{first}), and then isolate $u$ in the resulting formula by writing it as the $g_J$-inner product of a tensor, the Euler-Lagrange equation at $p,$ and $u.$ Lemmas \ref{step}, \ref{gam} lead to 

\begin{align*}
d_J \widetilde{\mathcal{N}}(J)(u)&=4\Re \Big[\int_X \Big(h^{s\bar{m}}h^{j\bar{p}}h_{k\bar{q}} J^i_{\bar{j}} \frac{\partial ^{1,0} u^k_{\bar{s}}}{\partial w_i}+h^{i\bar{m}}h^{j\bar{p}}h_{k\bar{q}}\Big(\frac{\partial ^{0,1} u^s_{\bar{j}}}{\partial \bar{w}_i}J^k_s +\frac{\partial ^{0,1} u^{\bar{s}}_{\bar{j}}}{\partial \bar{w}_i}J^k_{\bar{s}}\Big)+\\
& h^{s\bar{m}}h^{j\bar{p}}h_{k\bar{q}} \frac{\partial^{0,1} u^k_{\bar{s}}}{\partial \bar{w}_i}J^{\bar{i}}_{\bar{j}}\Big)\overline{N^q_{\bar{m}\bar{p}}} \vol_h \Big](p)\\
&-4\Re \Big[\int_X h^{s\bar{m}}h^{j\bar{p}}h_{k\bar{q}} \Big(u^i_{\bar{s}} \frac{\partial ^{1,0} J^k_{\bar{j}}}{\partial w_i}+u^{\bar{i}}_{\bar{s}} \frac{\partial ^{0,1} J^k_{\bar{j}}}{\partial \bar{w}_i}\Big)\overline{N^q_{\bar{m}\bar{p}}} \vol_{g_{J+u}} \Big](p)+\\
&4\Re \Big[\int_X h^{i\bar{m}}h^{j\bar{p}}h_{k\bar{q}}\Big\langle u\Big[\frac{\partial^{0,1}}{\partial \bar{w}_i},J\frac{\partial^{0,1}}{\partial \bar{w}_j}\Big],dw_k^{1,0}\Big\rangle \overline{N^q_{\bar{m}\bar{p}}} \vol_{g_{J+u}}\Big](p)\\
&-\int_X \Big(u^k_s \big(2g'_{k\bar{j}} N^v_{\bar{i}\bar{j}} \overline{N^v_{\bar{i}\bar{s}}}-g'_{k\bar{m}} N^s_{\bar{i}\bar{j}} \overline{N^m_{\bar{i}\bar{j}}}\big)+u^{\bar{k}}_s \big(2g'_{\bar{k}\bar{j}} N^v_{\bar{i}\bar{j}} \overline{N^v_{\bar{i}\bar{s}}}-g'_{\bar{k}\bar{m}} N^s_{\bar{i}\bar{j}} \overline{N^m_{\bar{i}\bar{j}}}\big)+\\
&u^k_{\bar{s}} \big(2g'_{km} N^v_{\bar{i}\bar{s}} \overline{N^v_{\bar{i}\bar{m}}}-g'_{kv} N^v_{\bar{i}\bar{j}} \overline{N^s_{\bar{i}\bar{j}}}\big)+u^{\bar{k}}_{\bar{s}} \big(2g'_{\bar{k}m} N^v_{\bar{i}\bar{s}} \overline{N^v_{\bar{i}\bar{m}}}-g'_{\bar{k}v} N^v_{\bar{i}\bar{j}} \overline{N^s_{\bar{i}\bar{j}}}\big)\Big)dV+\\
&\int_X \big(u^k_s g'_{k\bar{s}}+u^{\bar{k}}_s g'_{\bar{k}\bar{s}} + u^k_{\bar{s}} g'_{ks}+u^{\bar{k}}_{\bar{s}} g'_{\bar{k}s} \big)|N^v_{\bar{i}\bar{j}}|^2 dV \\
&=4\Re \Big[\int_X \Big(\overline{\tau'_{s\bar{m}i}} J^i_{\bar{j}} \overline{N^k_{\bar{m}\bar{j}}}+\overline{\tau'_{j\bar{p}i}} J^i_{\bar{j}} \overline{N^k_{\bar{s}\bar{p}}}-\tau_{k\bar{q}i} J^i_{\bar{j}} \overline{N^q_{\bar{s}\bar{j}}}-\frac{\partial^{1,0} J^i_{\bar{j}}}{\partial w_i} \overline{N^k_{\bar{s}\bar{j}}}-J^i_{\bar{j}} \frac{\partial^{1,0} \overline{N^k_{\bar{s}\bar{j}}}}{\partial w_i}\\
&-J^i_{\bar{j}} \overline{N^k_{\bar{s}\bar{j}}} \big(\sum_{c=1}^n \tau_{c\bar{c}i}\big)\Big) u^k_{\bar{s}} dV +\int_X \Big(\overline{\tau_{i\bar{m}i}} J^k_s \overline{N^k_{\bar{m}\bar{j}}}+\overline{\tau_{j\bar{p}i}} J^k_s \overline{N^k_{\bar{i}\bar{p}}}-\tau'_{k\bar{q}i} J^k_s \overline{N^q_{\bar{i}\bar{j}}}\\
&-\frac{\partial^{0,1} J^k_s}{\partial \bar{w}_i} \overline{N^k_{\bar{i}\bar{j}}}-J^k_s \frac{\partial^{0,1} \overline{N^k_{\bar{i}\bar{j}}}}{\partial \bar{w}_i}-J^k_s \overline{N^k_{\bar{i}\bar{j}}} \big(\sum_{c=1}^n \tau'_{c\bar{c}i}\big)\Big) u^s_{\bar{j}} dV + \int_X \Big(\overline{\tau_{i\bar{m}i}} J^k_{\bar{s}} \overline{N^k_{\bar{m}\bar{j}}}+\\
&\overline{\tau_{j\bar{p}i}} J^k_{\bar{s}} \overline{N^k_{\bar{i}\bar{p}}}-\tau'_{k\bar{q}i} J^k_{\bar{s}} \overline{N^q_{\bar{i}\bar{j}}}-\frac{\partial^{0,1} J^k_{\bar{s}}}{\partial \bar{w}_i} \overline{N^k_{\bar{i}\bar{j}}}-J^k_{\bar{s}} \frac{\partial^{0,1} \overline{N^k_{\bar{i}\bar{j}}}}{\partial \bar{w}_i}-J^k_{\bar{s}} \overline{N^k_{\bar{i}\bar{j}}} \big(\sum_{c=1}^n \tau'_{c\bar{c}i}\big)\Big) u^{\bar{s}}_{\bar{j}} dV +\\
& \int_X \Big(\overline{\tau_{s\bar{m}i}} J^{\bar{i}}_{\bar{j}} \overline{N^k_{\bar{m}\bar{j}}}+\overline{\tau_{j\bar{p}i}} J^{\bar{i}}_{\bar{j}} \overline{N^k_{\bar{s}\bar{p}}}-\tau'_{k\bar{q}i} J^{\bar{i}}_{\bar{j}} \overline{N^q_{\bar{s}\bar{j}}}-\frac{\partial^{0,1} J^{\bar{i}}_{\bar{j}}}{\partial \bar{w}_i} \overline{N^k_{\bar{s}\bar{j}}}-J^{\bar{i}}_{\bar{j}}\frac{\partial^{0,1} \overline{N^k_{\bar{s}\bar{j}}}}{\partial \bar{w}_i}\\
&-J^{\bar{i}}_{\bar{j}} \overline{N^k_{\bar{s}\bar{j}}} \big(\sum_{c=1}^n \tau'_{c\bar{c}i}\big)\Big) u^k_{\bar{s}} dV\Big]-4\Re \Big[\int_X \Big(u^k_{\bar{s}} \frac{\partial^{1,0} J^i_{\bar{j}}}{\partial w_k}+u^{\bar{k}}_{\bar{s}} \frac{\partial^{0,1} J^i_{\bar{j}}}{\partial \bar{w}_k} \Big)\overline{N^i_{\bar{s}\bar{j}}}dV \Big]\\
&+4\Re \Big[\int_X \Big(\frac{\partial^{0,1} J^s_{\bar{j}}}{\partial \bar{w}_i}u^k_s+\frac{\partial^{0,1} J^{\bar{s}}_{\bar{j}}}{\partial \bar{w}_i}u^k_{\bar{s}}\Big)\overline{N^k_{\bar{i}\bar{j}}} dV\Big]\\
&-\int_X \Big(u^k_s \big(2g'_{k\bar{j}} N^v_{\bar{i}\bar{j}} \overline{N^v_{\bar{i}\bar{s}}}-g'_{k\bar{m}} N^s_{\bar{i}\bar{j}} \overline{N^m_{\bar{i}\bar{j}}}\big)+u^{\bar{k}}_s \big(2g'_{\bar{k}\bar{j}} N^v_{\bar{i}\bar{j}} \overline{N^v_{\bar{i}\bar{s}}}-g'_{\bar{k}\bar{m}} N^s_{\bar{i}\bar{j}} \overline{N^m_{\bar{i}\bar{j}}}\big)+\\
&u^k_{\bar{s}} \big(2g'_{km} N^v_{\bar{i}\bar{s}} \overline{N^v_{\bar{i}\bar{m}}}-g'_{kv} N^v_{\bar{i}\bar{j}} \overline{N^s_{\bar{i}\bar{j}}}\big)+u^{\bar{k}}_{\bar{s}} \big(2g'_{\bar{k}m} N^v_{\bar{i}\bar{s}} \overline{N^v_{\bar{i}\bar{m}}}-g'_{\bar{k}v} N^v_{\bar{i}\bar{j}} \overline{N^s_{\bar{i}\bar{j}}}\big)\Big)dV+\\
&\int_X \big(u^k_s g'_{k\bar{s}}+u^{\bar{k}}_s g'_{\bar{k}\bar{s}} + u^k_{\bar{s}} g'_{ks}+u^{\bar{k}}_{\bar{s}} g'_{\bar{k}s} \big)|N^v_{\bar{i}\bar{j}}|^2 dV \\
&=\Re\Big\{\int_X \Big[\Big(4 \Big(\big(\overline{\tau'_{s\bar{m}i}} J^i_{\bar{j}}+\overline{\tau_{s\bar{m}i}} J^{\bar{i}}_{\bar{j}}\big)\overline{N^k_{\bar{m}\bar{j}}}+\big(\overline{\tau'_{j\bar{m}i}} J^i_{\bar{j}}+\overline{\tau_{j\bar{m}i}} J^{\bar{i}}_{\bar{j}}\big)\overline{N^k_{\bar{s}\bar{m}}}\\
&-\big(\tau_{k\bar{m}i} J^i_{\bar{j}}+\tau'_{k\bar{m}i} J^{\bar{i}}_{\bar{j}}\big)\overline{N^m_{\bar{s}\bar{j}}}-\Big(\frac{\partial^{1,0} J^i_{\bar{j}}}{\partial w_i}+\frac{\partial^{0,1} J^{\bar{i}}_{\bar{j}}}{\partial \bar{w}_i} \Big)\overline{N^k_{\bar{s}\bar{j}}}-J^i_{\bar{j}}\frac{\partial^{1,0} \overline{N^k_{\bar{s}\bar{j}}}}{\partial w_i}\\
&-J^{\bar{i}}_{\bar{j}}\frac{\partial^{0,1} \overline{N^k_{\bar{s}\bar{j}}}}{\partial \bar{w}_i}-\overline{N^k_{\bar{s}\bar{j}}}\Big(J^i_{\bar{j}}\big(\sum_{c=1}^n \tau_{c\bar{c}i}\big)+J^{\bar{i}}_{\bar{j}} \big(\sum_{c=1}^n \tau'_{c\bar{c}i}\big)\Big)+\overline{\tau_{i\bar{m}i}} J^j_k \overline{N^j_{\bar{m}\bar{s}}}+ \\
&\overline{\tau_{s\bar{m}i}} J^j_k \overline{N^j_{\bar{i}\bar{m}}}-\tau_{j\bar{m}i} J^j_k \overline{N^m_{\bar{i}\bar{s}}}-\frac{\partial^{0,1} J^j_k}{\partial \bar{w}_i}\overline{N^j_{\bar{i}\bar{s}}}-J^j_k \frac{\partial^{0,1} \overline{N^j_{\bar{i}\bar{s}}}}{\partial \bar{w}_i}-J^j_k \overline{N^j_{\bar{i}\bar{s}}}\big(\sum_{c=1}^n \tau'_{c\bar{c}i}\big)\\
&-\frac{\partial^{1,0} J^i_{\bar{j}}}{\partial w_k}\overline{N^i_{\bar{s}\bar{j}}}+\frac{\partial^{0,1} J^{\bar{s}}_{\bar{j}}}{\partial \bar{w}_i}\overline{N^k_{\bar{i}\bar{j}}}\Big)-\big(2g'_{km} N^v_{\bar{i}\bar{s}} \overline{N^v_{\bar{i}\bar{m}}}-g'_{kv} N^v_{\bar{i}\bar{j}} \overline{N^s_{\bar{i}\bar{j}}}\big)+g'_{ks} |N^v_{\bar{i}\bar{j}}|^2 \Big)u^k_{\bar{s}}+\\
&\Big(4\Big(\overline{\tau_{i\bar{m}i}}J^j_{\bar{k}} \overline{N^j_{\bar{m}\bar{s}}}+ \overline{\tau_{s\bar{m}i}}J^j_{\bar{k}} \overline{N^j_{\bar{i}\bar{m}}}-\tau'_{j\bar{m}i} J^j_{\bar{k}} \overline{N^m_{\bar{i}\bar{s}}}-\frac{\partial^{0,1} J^j_{\bar{k}}}{\partial \bar{w}_i}\overline{N^j_{\bar{i}\bar{s}}}-J^j_{\bar{k}}\frac{\partial^{0,1} \overline{N^j_{\bar{i}\bar{s}}}}{\partial \bar{w}_i}\\
&-J^j_{\bar{k}}\overline{N^j_{\bar{i}\bar{s}}}\big(\sum_{c=1}^n \tau'_{c\bar{c}i}\big)-\frac{\partial^{0,1} J^i_{\bar{j}}}{\partial \bar{w}_k}\overline{N^i_{\bar{s}\bar{j}}}\Big)-\big(2g'_{\bar{k}m} N^v_{\bar{i}\bar{s}} \overline{N^v_{\bar{i}\bar{m}}}-g'_{\bar{k}v} N^v_{\bar{i}\bar{j}} \overline{N^s_{\bar{i}\bar{j}}}\big)+\\
&g'_{\bar{k}s} |N^v_{\bar{i}\bar{j}}|^2\Big)u^{\bar{k}}_{\bar{s}}+\Big(-\big(2g'_{\bar{k}\bar{j}} N^v_{\bar{i}\bar{j}} \overline{N^v_{\bar{i}\bar{s}}}-g'_{\bar{k}\bar{m}} N^s_{\bar{i}\bar{j}} \overline{N^m_{\bar{i}\bar{j}}}\big)+g'_{\bar{k}\bar{s}} |N^v_{\bar{i}\bar{j}}|^2 \Big)u^{\bar{k}}_s+\Big(4\frac{\partial^{0,1} J^s_{\bar{j}}}{\partial \bar{w}_i}\overline{N^k_{\bar{i}\bar{j}}}\\
&-\big(2g'_{k\bar{j}} N^v_{\bar{i}\bar{j}} \overline{N^v_{\bar{i}\bar{s}}}-g'_{k\bar{m}} N^s_{\bar{i}\bar{j}} \overline{N^m_{\bar{i}\bar{j}}}\big)+g'_{k\bar{s}} |N^v_{\bar{i}\bar{j}}|^2 \Big)u^k_s\Big] dV \Big\}
\end{align*}

We may use $g_J(p)=I$ to induce the inner product $
\langle T,V\rangle=\sum_{s,k=1}^n (T^k_s \overline{V^k_s}+T^{\bar{k}}_s \overline{V^{\bar{k}}_s}+T^k_{\bar{s}} \overline{V^k_{\bar{s}}}+T^{\bar{k}}_{\bar{s}} \overline{V^{\bar{k}}_{\bar{s}}})$ of any $T,V \in C^{\infty}\big(X,\End(T^{\mathbb{C}}_{X,p})\big).$ Consider now the tensor $\widetilde{\mathcal{T}}_J=\widetilde{\mathcal{T}}^q_p dw_p^{1,0} \otimes \frac{\partial^{1,0}}{\partial w_q}+\widetilde{\mathcal{T}}^{\bar{q}}_p dw_p^{1,0} \otimes \frac{\partial^{0,1}}{\partial \bar{w}_q}+\widetilde{\mathcal{T}}^{q}_{\bar{p}} d\bar{w}_p^{0,1} \otimes \frac{\partial^{1,0}}{\partial w_q}+\widetilde{\mathcal{T}}^{\bar{q}}_{\bar{p}} d\bar{w}_p^{0,1} \otimes \frac{\partial^{0,1}}{\partial \bar{w}_q},$ where 
\begin{equation*}
\begin{split}
\widetilde{\mathcal{T}}^q_p &=4\frac{\partial^{0,1} J^p_{\bar{j}}}{\partial \bar{w}_i}\overline{N^q_{\bar{i}\bar{j}}}-\big(2g'_{q\bar{j}} N^v_{\bar{i}\bar{j}} \overline{N^v_{\bar{i}\bar{p}}}-g'_{q\bar{m}} N^p_{\bar{i}\bar{j}} \overline{N^m_{\bar{i}\bar{j}}}\big)+g'_{q\bar{p}} |N^v_{\bar{i}\bar{j}}|^2,
\end{split}
\end{equation*}

\begin{equation*}
\begin{split}
\widetilde{\mathcal{T}}^{\bar{q}}_p&=-\big(2g'_{\bar{q}\bar{j}} N^v_{\bar{i}\bar{j}} \overline{N^v_{\bar{i}\bar{p}}}-g'_{\bar{q}\bar{m}} N^p_{\bar{i}\bar{j}} \overline{N^m_{\bar{i}\bar{j}}}\big)+g'_{\bar{q}\bar{p}} |N^v_{\bar{i}\bar{j}}|^2 ,
\end{split}
\end{equation*}
\begin{equation*}
\begin{split}
\widetilde{\mathcal{T}}^{q}_{\bar{p}}&=4 \Big(\big(\overline{\tau'_{p\bar{m}i}} J^i_{\bar{j}}+\overline{\tau_{p\bar{m}i}} J^{\bar{i}}_{\bar{j}}\big)\overline{N^q_{\bar{m}\bar{j}}}+\big(\overline{\tau'_{j\bar{m}i}} J^i_{\bar{j}}+\overline{\tau_{j\bar{m}i}} J^{\bar{i}}_{\bar{j}}\big)\overline{N^q_{\bar{p}\bar{m}}}\\
&-\big(\tau_{q\bar{m}i} J^i_{\bar{j}}+\tau'_{q\bar{m}i} J^{\bar{i}}_{\bar{j}}\big)\overline{N^m_{\bar{p}\bar{j}}}-\Big(\frac{\partial^{1,0} J^i_{\bar{j}}}{\partial w_i}+\frac{\partial^{0,1} J^{\bar{i}}_{\bar{j}}}{\partial \bar{w}_i} \Big)\overline{N^q_{\bar{p}\bar{j}}}-J^i_{\bar{j}}\frac{\partial^{1,0} \overline{N^q_{\bar{p}\bar{j}}}}{\partial w_i}\\
&-J^{\bar{i}}_{\bar{j}}\frac{\partial^{0,1} \overline{N^q_{\bar{p}\bar{j}}}}{\partial \bar{w}_i}-\overline{N^q_{\bar{p}\bar{j}}}\Big(J^i_{\bar{j}}\big(\sum_{c=1}^n \tau_{c\bar{c}i}\big)+J^{\bar{i}}_{\bar{j}} \big(\sum_{c=1}^n \tau'_{c\bar{c}i}\big)\Big)+\overline{\tau_{i\bar{m}i}} J^j_q \overline{N^j_{\bar{m}\bar{p}}}+ \\
&\overline{\tau_{p\bar{m}i}} J^j_q \overline{N^j_{\bar{i}\bar{m}}}-\tau'_{j\bar{m}i} J^j_q \overline{N^m_{\bar{i}\bar{p}}}-\frac{\partial^{0,1} J^j_q}{\partial \bar{w}_i}\overline{N^j_{\bar{i}\bar{p}}}-J^j_q \frac{\partial^{0,1} \overline{N^j_{\bar{i}\bar{p}}}}{\partial \bar{w}_i}-J^j_q \overline{N^j_{\bar{i}\bar{p}}}\big(\sum_{c=1}^n \tau'_{c\bar{c}i}\big)\\
&-\frac{\partial^{1,0} J^i_{\bar{j}}}{\partial w_q}\overline{N^i_{\bar{p}\bar{j}}}+\frac{\partial^{0,1} J^{\bar{p}}_{\bar{j}}}{\partial \bar{w}_i}\overline{N^q_{\bar{i}\bar{j}}}\Big)-\big(2g'_{qm} N^v_{\bar{i}\bar{p}} \overline{N^v_{\bar{i}\bar{m}}}-g'_{qv} N^v_{\bar{i}\bar{j}} \overline{N^p_{\bar{i}\bar{j}}}\big)+\\
& g'_{qp} |N^v_{\bar{i}\bar{j}}|^2 ,
\end{split}
\end{equation*}
and where
\begin{equation*}
\begin{split}
\widetilde{\mathcal{T}}^{\bar{q}}_{\bar{p}} &=4\Big(\overline{\tau_{i\bar{m}i}}J^j_{\bar{q}} \overline{N^j_{\bar{m}\bar{p}}}+ \overline{\tau_{p\bar{m}i}}J^j_{\bar{q}} \overline{N^j_{\bar{i}\bar{m}}}-\tau'_{j\bar{m}i} J^j_{\bar{q}} \overline{N^m_{\bar{i}\bar{p}}}-\frac{\partial^{0,1} J^j_{\bar{q}}}{\partial \bar{w}_i}\overline{N^j_{\bar{i}\bar{p}}}-J^j_{\bar{q}}\frac{\partial^{0,1} \overline{N^j_{\bar{i}\bar{p}}}}{\partial \bar{w}_i}-J^j_{\bar{q}}\overline{N^j_{\bar{i}\bar{p}}}\big(\sum_{c=1}^n \tau'_{c\bar{c}i}\big)\\
&-\frac{\partial^{0,1} J^i_{\bar{j}}}{\partial \bar{w}_q}\overline{N^i_{\bar{p}\bar{j}}}\Big)-\big(2g'_{\bar{q}m} N^v_{\bar{i}\bar{p}} \overline{N^v_{\bar{i}\bar{m}}}-g'_{\bar{q}v} N^v_{\bar{i}\bar{j}} \overline{N^p_{\bar{i}\bar{j}}}\big)+g'_{\bar{q}p} |N^v_{\bar{i}\bar{j}}|^2.
\end{split}
\end{equation*}
Then, \[d_J \widetilde{\mathcal{N}}(J)(u)=\Re \Big\{\int_X \langle \widetilde{\mathcal{T}}_J, \overline{u}\rangle_I dV\Big\},\] and so $\widetilde{\mathcal{T}}_J=0$ is the Euler-Lagrange equation at $p.$ 

We can rewrite $\mathcal{T}^{\bar{q}}_p,$ and $\mathcal{T}^q_{\bar{p}}$ somewhat more meaningfully, using that $J\Big(\frac{\partial^{0,1}}{\partial \bar{w}_j}\Big)=J^i_{\bar{j}} \frac{\partial^{1,0}}{\partial w_i}+J^{\bar{i}}_{\bar{j}}\frac{\partial^{0,1}}{\partial \bar{w}_i},$ and the hermitian nature of the fundamental form $\omega$: 

\begin{equation*}
\begin{split}
\mathcal{T}^{\bar{q}}_{\bar{p}} &=4\Big[J^j_{\bar{q}} \Big(\frac{\partial^{0,1} \omega_{m\bar{i}}}{\partial \bar{w}_i} \overline{N^j_{\bar{m}\bar{p}}}+\frac{\partial^{0,1} \omega_{m\bar{p}}}{\partial \bar{w}_i} \overline{N^j_{\bar{i}\bar{m}}}-\frac{\partial^{0,1} \omega_{j\bar{m}}}{\partial \bar{w}_i} \overline{N^m_{\bar{i}\bar{p}}}-\sum_{c=1}^n \frac{\partial^{0,1} \omega_{c\bar{c}}}{\partial \bar{w}_i} \overline{N^m_{\bar{p}\bar{j}}}\Big)\\
&-\frac{\partial^{0,1}}{\partial \bar{w}_i} \big(J^j_{\bar{q}} \overline{N^j_{\bar{i}\bar{p}}}\big)-\frac{\partial^{0,1} J^{\bar{i}}_{\bar{j}}}{\partial \bar{w}_q} \overline{N^i_{\bar{p}\bar{j}}}\Big]-\big(2g'_{\bar{q}m} N^v_{\bar{i}\bar{p}} \overline{N^v_{\bar{i}\bar{m}}}-g'_{\bar{q}v} N^v_{\bar{i}\bar{j}} \overline{N^p_{\bar{i}\bar{j}}}\big)+g'_{\bar{q}p} |N^v_{\bar{i}\bar{j}}|^2,
\end{split}
\end{equation*}
\begin{equation*}
\begin{split}
\mathcal{T}^q_{\bar{p}}&=4\Big[\Big(J\Big(\frac{\partial^{0,1}}{\partial \bar{w}_j}\Big)\omega_{m\bar{p}}\Big)\overline{N^q_{\bar{m}\bar{j}}}+\Big(J\Big(\frac{\partial^{0,1}}{\partial \bar{w}_j}\Big)\omega_{m\bar{j}}\Big)\overline{N^q_{\bar{p}\bar{m}}}-\Big(J\Big(\frac{\partial^{0,1}}{\partial \bar{w}_j}\Big)\omega_{q\bar{m}}\Big)\overline{N^m_{\bar{p}\bar{j}}}-\frac{\partial^{1,0}}{\partial w_i} \big(J^i_{\bar{j}}\overline{N^q_{\bar{p}\bar{j}}}\big)\\
&-\frac{\partial^{0,1}}{\partial \bar{w}_i}\big(J^{\bar{i}}_{\bar{j}}\overline{N^q_{\bar{p}\bar{j}}}\big)-\Big(J\Big(\frac{\partial^{0,1}}{\partial \bar{w}_j}\Big)\Big(\sum_{c=1}^n \omega_{c\bar{c}} \Big)\Big)\overline{N^q_{\bar{p}\bar{j}}}+J^j_q \Big(\frac{\partial^{0,1} \omega_{m\bar{i}}}{\partial \bar{w}_i}\overline{N^j_{\bar{m}\bar{p}}}+\frac{\partial^{0,1} \omega_{m\bar{p}}}{\partial \bar{w}_i}\overline{N^j_{\bar{i}\bar{m}}}\\
&-\frac{\partial^{0,1} \omega_{j\bar{m}}}{\partial \bar{w}_i}\overline{N^m_{\bar{i}\bar{p}}}-\sum_{c=1}^n \frac{\partial^{0,1} \omega_{c\bar{c}}}{\partial \bar{w}_i}\overline{N^j_{\bar{i}\bar{p}}}\Big)-\frac{\partial^{0,1}}{\partial \bar{w}_i}\big(J^j_q \overline{N^j_{\bar{i}\bar{p}}}\big)-\frac{\partial^{1,0} J^i_{\bar{j}}}{\partial w_q}\overline{N^i_{\bar{p}\bar{j}}}+\frac{\partial^{0,1} J^{\bar{p}}_{\bar{j}}}{\partial \bar{w}_i} \overline{N^q_{\bar{i}\bar{j}}}\Big]\\
&-\big(2g'_{qm} N^v_{\bar{i}\bar{p}} \overline{N^v_{\bar{i}\bar{m}}}-g'_{qv} N^v_{\bar{i}\bar{j}} \overline{N^p_{\bar{i}\bar{j}}}\big)+g'_{qp} |N^v_{\bar{i}\bar{j}}|^2. 
\end{split}
\end{equation*}
In the case of $\mathcal{N},$ there is no $d_{\vol}(\|N_J\|^2_{g_J} \vol_{g_J})\cdot d_J(\vol_{g_J})(u)$ term in the first variation, neither are there derivatives of the Riemannian volume form. We can recover the Euler-Lagrange equation of $\mathcal{N}$ at $p$ from that of $\widetilde{\mathcal{N}}$ to find that it is a tensor equation $\mathcal{T}_{\mathcal{N}}=0,$ where the components of $\mathcal{T}_{\mathcal{N}}$ are of the claimed form.
\end{proof}

All integrable almost complex structures on $X$ are critical points of both $\mathcal{N},$ and $\widetilde{\mathcal{N}},$ but they are likely not the only ones. It is also unclear if $\widetilde{\mathcal{N}}$ has any advantages over $\mathcal{N}.$ It might make sense to try eliminating the derivatives of $\overline{N_J}$ that appear in the Euler-Lagrange equation of $\mathcal{N}$ since they do not directly contain information about integrability.

\noindent
Gabriella Clemente                                

\noindent
e-mail: clemente6171@gmail.com
\end{document}